\documentclass[a4paper]{article}
\usepackage[a4paper,top=4cm,bottom=3cm,left=3cm,right=3cm,marginparwidth=2cm]{geometry}

\usepackage{amsmath}
\usepackage{amssymb}
\usepackage{multirow}
\usepackage{mathtools}
\usepackage{amsthm}
\usepackage{graphicx}
\usepackage[colorinlistoftodos]{todonotes}
\usepackage[colorlinks=true, allcolors=blue]{hyperref}
\DeclareMathOperator{\PG}{PG}

\usepackage{pgf,tikz,pgfplots, booktabs,subfig}
\usepackage{mathrsfs}
\usetikzlibrary{arrows}

\usepackage[english]{babel}
\usepackage[utf8x]{inputenc}
\usepackage[T1]{fontenc}

\newtheorem{definition}{Definition}[section]

\newtheorem{theorem}[definition]{Theorem}
\newtheorem{lemma}[definition]{Lemma}
\newtheorem{gevolg}[definition]{Corollary}
\newtheorem{result}[definition]{Result}

\newtheorem{remark}[definition]{Remark}

\newtheorem{proposition}[definition]{Proposition}

\def\F{\mathbb{F}}
\def\C{\mathcal{C}}
\def\S{\mathcal{S}}
\def\P{\mathcal{P}}
\def\D{\mathcal{D}}
\def\Q{\mathcal{Q}}
\def\A{\mathcal{A}}

\def\PGL{\mathrm{PGL}}

\title{Translation hyperovals and $\F_2$-linear sets \\of pseudoregulus type}
\author{Jozefien D'haeseleer \thanks{This author is a PhD fellow of FWO (Research foundation -- Flanders, Belgium)} \and Geertrui Van de Voorde }
\begin{document}
\maketitle
\begin{abstract}
In this paper, we study translation hyperovals in $\PG(2,q^k)$. The main result of this paper characterises the point sets defined by translation hyperovals in the Andr\'e/Bruck-Bose representation. We show that the affine point sets of translation hyperovals in $\PG(2,q^k)$ are precisely those that have a scattered $\F_2$-linear set of pseudoregulus type in $\PG(2k-1,q)$ as set of directions. This  correspondence is used to generalise the results of Barwick and Jackson who provided a characterisation for translation hyperovals in $\PG(2,q^2)$, see \cite{BJeven}.

\end{abstract}
\section{Introduction}
Let $\PG(n,q)$ denote the $n$-dimensional projective space over the finite field $\mathbb{F}_q$ with $q$ elements.  

A {$k$-arc} in $\PG(2,q)$ is a set of $k$ points such that no three of them are collinear. A \emph{hyperoval} in $\PG(2,q)$ is a $(q+2)$-arc. Hyperovals only exist when $q$ is even. A {\em translation hyperoval} is a hyperoval $H$ such that there exists a bisecant $\ell$ of $H$ such that the group of elations with axis $\ell$ acts transitively on the points of $H$ not on $\ell$. It is well-known (see e.g. \cite[Theorem 8.5.4]{hirshfeld}) that every translation hyperoval in $\PG(2,q)$ is $\PGL$-equivalent to a point set $\{(1,t,t^{2^i})|t\in \F_{q}\}\cup \{(0,1,0),(0,0,1)\},$ where $q=2^h$ and $\gcd(i,h)=1$.

In \cite{BJeven}, Barwick and Jackson provided a chararacterisation of translation hyperovals in $\PG(2,q^2)$: they considered a set $\C$ of points in $\PG(4,q)$, $q$ even, with certain combinatorial properties with respect to the planes of $\PG(4,q)$ (see Section \ref{veralgemening} for details). They proved that the set $\C'$ of directions determined by the points of $\C$ has the property that every line  intersects $\C'$ in $0,1,3$ or $q-1$ points. They then used this to construct a Desarguesian line spread $\S$ in $\PG(3,q)$, such that in the corresponding Andr\'e/Bruck-Bose plane $\P(\S)\cong \PG(2,q^2)$, the points corresponding to $\C$ form a translation hyperoval. This extended the work done in \cite{BJodd}, where the same authors gave a similar characterisation of Andr\'e/Bruck-Bose representation of conics for $q$ odd. 

In this paper, we will generalise the combinatorial characterisation provided by Barwick and Jackson for translation hyperovals. In order to do this, we prove our main theorem, linking translation hyperovals with $\F_2$-linear sets of pseudoregulus type:
\begin{theorem} \label{main} Let $\Q$ be a set of $q^k$ affine points in $\PG(2k,q)$, $q=2^h$, $h\geq 4$, $k\geq 2$, determining a set $D$ of $q^k-1$ directions in the hyperplane at infinity $H_\infty=\PG(2k-1,q)$. Suppose that every line has $0$, $1$, $3$ or $q-1$ points in common with the point set $D$. Then 
\begin{itemize}
\item[(1)] $D$ is an $\F_2$-linear set of pseudoregulus type.
\item[(2)] There exists a Desarguesian spread $\mathcal{S}$ in $H_\infty$ such that, in the Bruck-Bose plane $\mathcal{P}(\mathcal{S})\cong \PG(2,q^k)$, with $H_\infty$ corresponding to the line $l_\infty$, the points of $\Q$ together with $2$ extra points on $\ell_\infty$, form a translation hyperoval in $\PG(2,q^k)$.
\end{itemize}

Vice versa, via the Andr\'e/Bruck-Bose construction, the set of affine points of a translation hyperoval in $\PG(2,q^k)$, $q> 4, k\geq 2$, corresponds to a set $\Q$ of $q^k$ affine points in $\PG(2k,q)$ whose set of determined directions $D$ is an $\F_2$-linear set of pseudoregulus type. Consequently, every line meets $D$ in $0,1,3$ or $q-1$ points.
\end{theorem}

This paper is organised as follows. In Section \ref{sectionprel} we give the necessary definitions and background, in section \ref{bewijs}, we provide a proof of Theorem \ref{main}. Finally, we use this result in Section \ref{veralgemening} to generalise the result of Barwick-Jackson \cite{BJeven}.

\section{Preliminaries}\label{sectionprel}
\subsection{Linear sets}\label{subsectionlinear}
Linear sets are a central object in finite geometry and have been studied intensively, mainly due to the connection with other objects such as semifield planes, blocking sets, and more recently, MRD codes. (see e.g. \cite{lavrauw4}, \cite{linset}, \cite{olga}).  
 
 Let $V$ be an $r$-dimensional vector space over $\mathbb{F}_{q^n}$, let $\Omega$ be the projective space $\PG(V) = \PG(r-1, q^n)$. A set $T$ is said to be an \emph{$\mathbb{F}_q$-linear set} of $\Omega$ of rank $t$ if it is defined by the non-zero vectors
of an $\mathbb{F}_q$-vector subspace $U$ of $V$ of dimension $t$, i.e. $$T = L_U = \{\langle u \rangle_{\mathbb{F}_{q^n}}| u \in U \setminus \{0\}\}.$$
 
The points of $\PG(r-1,q^n)$ correspond to $1$-dimensional subspaces of $\mathbb{F}_{q^n}^{r}$, and hence to $n$-dimensional subspaces of $\mathbb{F}_q^{rn}$. In this way, the point set of $\PG(r-1,q^n)$ corresponds to a set $\D$ of $(n-1)$-dimensional subspaces of $\PG(rn-1,q)$, which partitions the point set of $\PG(rn-1,q)$.  The set $\D$ is called a \emph{Desarguesian spread}, and we have a one-to-one correspondence between the points of $\PG(r-1,q^n)$ and the elements of $\D$. Using coordinates, we see that a point $P=(x_0,x_1,\dots x_{r-1})_{q^n} \in \PG(r-1,q^n)$ corresponds to the set $\{ (\alpha x_0,\alpha x_1, \dots, \alpha x_{r-1})_q| \alpha \in \mathbb{F}_{q^n}  \}$ in $\PG(rn-1,q)$. Note that we have used $r$ coordinates from $\F_{q^n}$, defined up to $\F_q$-scalar multiple to define points of $\PG(rn-1,q)$, and the set $\{ (\alpha x_0,\alpha x_1, \dots, \alpha x_{r-1})_q| \alpha \in \mathbb{F}_{q^n}  \}$ consists of $\frac{q^n-1}{q-1}$ different points forming an $(n-1)$-dimensional space. Hence, we find that $\D$ is given by the set of $(n-1)$-spaces 
$$\{ (\alpha x_0,\alpha x_1, \dots, \alpha x_{r-1})_q| \alpha \in \mathbb{F}_{q^n}  \} \; \mathrm{for\ all\ }(x_0,x_1,\ldots,x_{r-1})\in \F_{q^n}^r.$$

Note that these coordinates for points in $\PG(rn-1,q)$ can be transformed into the usual coordinates consisting of $rn$ elements of $\F_q$ by representing the elements of $\F_{q^n}$ as the $n$ coordinates with respect to a fixed basis of $\F_{q^n}$ over $\F_q$.

 We also have a more geometric perspective on the notion of a linear set; namely, an $\mathbb{F}_q$-linear set is a set $T$ of points of $\PG(r-1,q^n)$ for which there exists a subspace $\pi$ in $\PG(rn-1,q)$ such that the points of $T$ correspond to the elements of $\D$ that have a non-empty intersection with $\pi$. For more on this approach to linear sets, we refer to \cite{linset}.
If the subspace $\pi$ intersects each spread element in at most a point, then $\pi$ is called \emph{scattered} with respect to $\D$ and the associated linear set is called a {\em scattered} linear set.  

Note that if $\pi$ is $(n-1)$-dimensional and scattered, then the associated $\mathbb{F}_q$-linear set has rank $n$ and has exactly $\frac{q^n-1}{q-1}$ points, and conversely. In this paper, we will make use of the following bound on the rank of a scattered linear set.
\begin{result}[{\cite[Theorem 4.3]{lavrauw3}}]\label{thm1}
The rank of a scattered $\mathbb{F}_q$-linear set in $\PG(r-1,q^n)$ is at most $ rn/2$.
\end{result}

A \emph{maximum scattered linear set} is a scattered $\mathbb{F}_q$-linear set in $\PG(r-1,q^n)$ with rank  $rn/2$.
In this article we work with maximum scattered linear sets to which a geometric structure, called \emph{pseudoregulus}, can be associated. For more information, we refer to \cite{GVML} and \cite{italiaan}. \begin{definition}
Let $S$ be a scattered $\mathbb{F}_q$-linear set of $\PG(2k-1,q^n)$ of rank $kn$, where $n, k\geq 2$. We say that $S$ is of \emph{pseudoregulus type} if
\begin{enumerate}
    \item  there exist $m=\frac{q^{nk}-1}{q^n-1}$ pairwise disjoint lines of $\PG(2k-1, q^n)$, say $s_1, s_2, \dots , s_m$, such that
    \begin{align*}
    |  S\cap  s_i|=\frac{q^n-1}{q-1}\quad \forall i=1,\dots,m,
    \end{align*}
    \item  there exist exactly two $(k-1)$-dimensional subspaces $T_1$ and $T_2$ of $\PG(2k-1,q^n)$ disjoint from $S$ such that $T_j\cap s_i \neq \emptyset$ for each $i = 1,\dots, m$ and $j = 1, 2$.
\end{enumerate}
\end{definition}
The set of lines $s_i$, $i=1,\dots m$ is called the {\em pseudoregulus} of $\PG(2k-1,q^n)$ associated with the linear set $S$ and we refer to $T_1$ and $T_2$ as \emph{transversal spaces} to this pseudoregulus. Since a maximum scattered linear set spans the whole space, we see that the transversal spaces are disjoint. 

 Throughout this paper we need the following result of \cite{italiaan} on pseudoreguli . Applied to $\F_2$-linear sets, this gives us the following result.
    \begin{result}[{\cite[Theorem 3.12]{italiaan}}]\label{thmitaliaan}
    Each $\mathbb{F}_2$-linear set of $\PG(2k-1, q
)$, $q$ even, of pseudoregulus type, is of the form $L_{\rho,f}$ with 
\begin{align*}
    L_{\rho,f} = \{ ( u, \rho f(u))_q| u\in U_0 \},
\end{align*}
with $\rho\in \F_{q}^*$, $U_0,U_\infty$ the $k$-dimensional vector spaces corresponding to the transversal spaces $T_0,T_\infty$ and with $f:U_0\rightarrow U_\infty$ an invertible semilinear map with companion automorphism $\sigma\in Aut(\mathbb{F}_q)$, $Fix(\sigma)=\{0,1\}$.
    \end{result}
Note that in the previous result, $\PG(2k-1,q)$ is identified with $\PG(V)$, $V= U_0 \oplus U_\infty$ and a point, corresponding to a vector $v=v_1+v_2\in U_0 \oplus U_\infty$, has coordinates $(v_1,v_2)_q$. 


\subsection{The Barlotti-Cofman and Andr\'e/Bruck-Bose constructions}
   

In this paper, we will switch between three different representations of a projective plane $\PG(2,q^k)$, $q=2^h$. Using the Andr\'e/Bruck-Bose correspondence, we can, on one hand, model this plane as a subset of points and $k$-spaces in $\PG(2k,q)$, determined by a $(k-1)$-spread at infinity. On the other hand, we can see it as a subset of points and $hk$-spaces of $\PG(2hk,2)$ determined by a $(hk-1)$-spread at infinity. We can switch between the $\PG(2k,q)$-setting an the $\PG(2hk,2)$-setting by the {\em Barlotti-Cofman} correspondence, which is a natural generalization of the Andr\'e/Bruck-Bose correspondence.


The Barlotti-Cofman representation of the projective space $\PG(2k,2^h)$ in $\PG(2hk,2)$ is defined as follows (see \cite{barlotticofman}).
Let $\S'$ be a Desarguesian $(h-1)$-spread in $\PG(2hk-1, 2)$. Embed $\PG(2hk-1, 2)$ as a hyperplane $\tilde{H_\infty}$ in $\PG(2hk, 2)$.
Consider the following incidence structure $\mathcal{P}(\S) = (\mathcal{P},\mathcal{L} )$, where incidence is natural:\\
\begin{itemize}
	\item The set $\mathcal{P}$ of points consists of the $2^{2hk}$ {\em affine points} $P_i$ in $\PG(2hk,2)$ (i.e. the points not in $\tilde{H_\infty}$) together with elements of the $(h-1)$-spread $\S'$ in $\tilde{H_\infty}$
	\item The set $\mathcal{L}$ of lines consist of the following two sets of subspaces in $\PG(2hk,2)$.
	\begin{itemize}
		\item The set of $h$-spaces spanned by an element of $\S'$ and an affine point of $\PG(2hk,2)$.
		\item The set of $(2h-1)$-spaces in $\tilde{H_\infty}$ spanned by two different elements of $\S'$.
	\end{itemize}
\end{itemize}
This incidence structure $(\mathcal{P},\mathcal{L} )$ is isomorphic to $\PG(2k,2^h)$. We use the notation $P_i$ for the affine point of $\PG(2k,2^h)$ (i.e. a point not contained in $H_\infty$) which corresponds to the affine point $\tilde{P_i}\in \PG(2hk,2)$. A point, say $R_i$ in $H_\infty$, corresponds to the element $\S'(R_i)$ of the $(h-1)$-spread $\S'$ in $\tilde{H_\infty}$.
As already mentioned above, during this paper we will work in the following three projective spaces:

\begin{itemize}
    \item The $2k$-dimensional projective space $\Pi_q=\PG(2k,q)$, $q=2^h, h>2$, with the $(2k-1)$-space at infinity called $H_\infty$.
    \item The projective plane $\Pi_{q^k}=\PG(2,q^k)$, $q=2^h$ with line  at infinity called $\ell_\infty$. Given a Desarguesian $(k-1)$-spread $\S$ in $H_\infty$ in $\Pi_q$, the plane $\Pi_{q^k}$ is obtained by the Andr\'e-Bruck-Bose construction using $\S$.
        \item The $2hk$-dimensional projective space $\Pi_2=\PG(2hk,2)$,  with the $(2hk-1)$-space $\tilde{H_\infty}$ at infinity.
        Note that the Barlotti-Cofman representation of $\Pi_q$ defines a Desarguesian $(h-1)$-spread $\S'$ in $\tilde{H_\infty}$. Moreover, if $\S$ is the $(k-1)$-spread in $H_\infty$ in $\Pi_q$ such that $\Pi_{q^ k}$ is the corresponding projective plane, the Andr\'e-Bruck-Bose representation of $\Pi_{q^k}$ in $\Pi_2$ gives rise to a Desarguesian $(hk-1)$-spread $\tilde{\S}$ in $\tilde{H_\infty}$, such that $\S'$ is a subspread of $\tilde{\S}$. 
\end{itemize}

\vspace{0.5cm}

\section{The proof of the main theorem} \label{bewijs}
Consider $\Pi_q=\PG(2k,q)$ and the hyperplane $H_\infty$ of $\PG(2k,q)$. Recall that a point of $\PG(2k,q)$ is called {\em affine} if it is not contained in $H_\infty$. Likewise, a line is called \emph{affine} if it is not contained in $H_\infty$.  Let $P_1,P_2$ be affine points, then the point $P_1P_2\cap H_\infty$ is the {\em direction} determined by the line $P_1P_2$. If $\Q$ is a set of affine points, then the {\em directions determined by $\Q$} are all points of $H_\infty$ that appear as the direction of a line $P_iP_j$ for some $P_i,P_j\in \Q$.

\vspace{0.5cm}
From now on, we consider a set $\Q$ satisfying the conditions of Theorem \ref{main}:
\begin{itemize}
	\item  $\Q$ is a set of $q^k$ affine points in $\PG(2k,q)$, $q=2^h$, $h\geq 4$, $k \geq 2$;
	\item   $D$, the set of directions determined by $\Q$ at the hyperplane at infinity $H_\infty$ has size $q^k-1$;
	\item  Every line  has $0$, $1$, $3$ or $q-1$ points in common with the point set $D$. 
\end{itemize}

\subsection{The $(q-1)$-secants to $D$ are disjoint}

\begin{definition}
	A \emph{$0$-point} in $H_\infty$ is a point $P\notin D$ such that $P$ is contained in at least one $(q-1)$-secant to $D$. 
\end{definition} From Proposition \ref{q-1disjunct}, it will follow that a $0$-point is contained in precisely one $(q-1)$-secant to $D$. We first start with two lemmas.

\begin{lemma}\label{no3collinear} No three points of $\Q$ are collinear.
\end{lemma}
\begin{proof}
	Let $l$ be an affine line in $\PG(2k,q)$ containing $3\leq t\leq q$ points of $\Q$, and let $P'=l\cap H_\infty$. A point $P_i\in \Q\setminus l$ determines a plane $\alpha_i=\langle P_i,l\rangle$ such that the line $l_i= \alpha_i\cap H_\infty$ is a $(q-1)$-secant: the lines through $P_i$ and a point of $l\cap \Q$ determine $t\geq 3$ directions of $D$ on the line $l_i$, different from the point $P' \in D$. So $l$ contains more than three points of $D$, showing that $l_i$ is a $(q-1)$-secant. Furthermore, the plane $\alpha_i$ contains at most $q$ affine points of $\Q$, as every affine line in $\alpha$ through a $0$-point of $l_i$ contains at most one element of $\Q$. 
	
	This implies that each of the $q^k-t\geq q^k-q$ points of $\Q\setminus l$ define a plane $\alpha$, with $\alpha \cap H_\infty$ a $(q-1)$-secant, and so that $\alpha$ contains at most $q-t\leq q-3$ points of $\Q\setminus l$. This shows that the number of such planes $\alpha_i$ through $l$, and hence the number of $(q-1)$-secants through $P'$, is at least $\frac{q^k-q}{q-3}$. This gives that there are at least $1+\frac{q^k-q}{q-3}(q-2)>q^k-1$ points of $D$, a contradiction.  
\end{proof}

\begin{lemma}\label{hulplemma} 
Let $\gamma$ be a plane in $\PG(2k,q)$ containing $4$ points $P_1,P_2,P_3$ and $P_4$ of $\Q$, such that $P_1P_2 \cap P_3P_4\notin \Q \cup D$. Then $\gamma$ meets $H_\infty$ in a $(q-1)$-secant to $D$.
\end{lemma}
\begin{proof}By Lemma \ref{no3collinear}, no three points of $P_1, P_2,P_3,P_4$ are collinear. Since $P_1P_2 \cap P_3P_4\notin D$, we see that $P_1P_2$ and $P_3P_4$ define two different directions in $H_\infty$.  The four points $P_1,P_2,P_3$ and $P_4$ determine at least $4$ directions on the line $\gamma \cap H_\infty$. The statement follows since a line  contains $0,1,3$ or $q-1$ points of $D$.
\end{proof}

\begin{proposition}\label{q-1disjunct} Every two $(q-1)$-secants to $D$ are disjoint.

\end{proposition}\begin{proof} Consider a point $P_0\in \Q$. Then, by Lemma \ref{no3collinear}, all points of $D$ are defined by the lines $P_0 P_i$ with $P_i\in \Q\setminus \{P_0\}$. Let $P_i'$ denote the direction of the line $P_0 P_i$, that is, the point $P_0P_i\cap H_\infty$. We see that a line through a point $P_i'\in D$ contains $0$ or $2$ points of $\Q$.

Let $l_\alpha$ and $l_\beta$ be two lines, both containing $q-1$ points of $D$, with $P'=l_\alpha \cap l_\beta$. Let $\alpha=\langle P_0, l_\alpha \rangle$ and $\beta=\langle P_0, l_\beta \rangle$  and let $\{P_{1 \alpha},P_{2 \alpha}\}$ and $\{P_{1 \beta},P_{2 \beta}\}$ be the $0$-points in $l_\alpha$ and $l_\beta$. Note that $P'$ may be amongst these points. It follows from the argument above that there are precisely $q$ points in $\alpha\cap \Q$ and that the affine points of $\Q$ in $\alpha$ together with the two points $P_{1 \alpha},P_{2 \alpha}$ form a hyperoval $H_\alpha$. Similarly, we find a hyperoval $H_\beta$ in $\beta$.

We first suppose that $P'\in D$. This implies that there is a point $P\neq P_0$ of $\Q$ on the line $P_0 P'$. Note that $P_0$ and $P$ are contained in $H_\alpha \cap H_\beta$. 

Consider a point $R\in l_\alpha$, different from $P', P_{1 \alpha},P_{2 \alpha}$. Then $R\in D$ and through $R$, there are $\frac{q}{2}$ bisecants to $H_\alpha\neq l_\alpha$. One of these bisecants contains $P$ and another one contains $P_0$. Since $q>8$, there exists a bisecant to $H_\alpha$ through $R$ which intersects the line $P_0 P$ in a point $R_0 \notin \{P_0,P,P'\}$. Through $R_0$, there are $\frac{q}{2}-2$ bisecants $r_i$ to $H_\beta$, different from the lines $R_0 P$, $R_0P_{1\beta}$ and $R_0P_{2\beta}$. Let $r_i\cap l_\beta=R_i, i=1,\dots,\frac{q}{2}-2$. A plane $\langle R, r_i \rangle$ contains two lines, $r_i$ and $m=R R_0$, both containing two points of $\Q$ and $r_i\cap m=R_0\notin \Q$. Hence, by Lemma \ref{hulplemma} we find that every line $R R_i$ is a $(q-1)$-secant to $D$.

So there are $\frac{q}{2}-2$ $(q-1)$-secants of the form $RR_i$, and the total number of $0$-points on these lines is $2(\frac{q}{2}-2)=q-4$. Let $\Omega$ be the set of these $0$-points. We call a \emph{$(\leq 3)$-secant} in $\langle l_\alpha, l_\beta \rangle$ a line with at most $3$ points of $D$. A line through $P'$ in $\langle l_\alpha,l_\beta \rangle$ intersects all lines $RR_i$. The $q-4$ points of $\Omega$ lie on the $q-1$ lines through $P'$ different from $l_\alpha$ and $l_\beta$. Since every line $RR_i$ contains precisely two $0$-points, we find that for $q>8$ there are at most $3$ $(\leq 3)$-secants through $P'$: if there are at least four $(\leq 3)$-secants through $P'$ in $\langle l_\alpha, l_\beta \rangle$, then there are at least $\frac{q}{2}-2-2$ $0$-points of $\Omega$ on each of these lines, as we supposed that $P'\in D$. This implies that there would be at least $4(\frac{q}{2}-4)>q-4$ $0$-points in $\Omega$, which gives a contradiction for $q\geq 16$.

Now we distinguish different cases depending on the number of $(\leq 3)$-secants through $P'$. In each of the cases we will show that there exists at least two $(\leq 3)$-secants $l_1,l_2$ in $\langle l_\alpha,l_\beta \rangle$, and a point $X\notin D$ not on these lines. This leads to a contradiction since there are at least $q+1-7$ lines through $X$, both intersecting $l_1$ and $l_2$ in a point not in $D$, and not through $l_1\cap l_2$. These lines contain at least $3$ points not in $D$ so they have to be $(\leq 3)$-secants. But this implies that there are at least $1+(q-6)(q-3)=q^2-9q+19$ points in $\langle l_\alpha,l_\beta \rangle$, not contained in $D$. On the other hand, there are at most three $(\leq 3)$-secants through $P'$ and the other lines through $P'$ contain two $0$-points. This implies that there are at most $3q+ 2(q-2)=5q-4<q^2-9q+19$ points in $\langle l_\alpha,l_\beta \rangle$, not contained in $D$. This gives a contradiction for $q\geq 16$.

It remains to show that in every case there exists at least two $(\leq 3)$-secants and a point $X\notin D$, not on these lines. 
\begin{itemize}
    \item Suppose first that there are two or three $(\leq 3)$-secants through $P'$. These lines are different from $l_\alpha$, so they do not contain the point $P_{1\alpha}$. Then $X=P_{1 \alpha}\notin D$ is a point not on the $(\leq 3)$-secants.
    \item Suppose there is an unique $(\leq 3)$-secant $l$ through $P'$. Then every other line through $P'$ contains two $0$-points. Suppose first that there exists a $0$-point $P_1$ so that $P_{1\alpha} P_1 \cap l \notin D$. Then $l'=P_{1\alpha} P_1$ contains $3$ points not in $D$, so $l'$ is a $(\leq 3)$-secant. Note that $P_1\neq P_{2\alpha}$ as otherwise $P_{1\alpha} P_1 \cap l=l_\alpha \cap l=P' \in D $. Hence $X=P_{2\alpha}\notin D$ is not contained in $l\cup l'$.
    
    If there is no point $P_1$ so that $P_{1 \alpha} P_1 \cap l \notin D$, then all $2q-4$ $0$-points on the $(q-1)$-secants through $P'$, different from $l_\alpha, l_\beta$, lie on at most $2$ lines  $P_{1\alpha}P_1$ and $P_{1\alpha}P_2$, with $P_1,P_2\in D\cap l\setminus \{P'\}$. But then $P_{1\alpha}P_1$ and $P_{1\alpha}P_2$ are $(\leq 3)$-secants. Note that these lines are different from $l_\alpha$, and so, they do not contain $P_{2\alpha}$. Hence we may take $X=P_{2\alpha}$.
    
    \item Suppose all lines through $P'$ are $(q-1)$-secants with $\Gamma$ the corresponding set of $2q+2$ $0$-points. Let $G\in \Gamma$ and consider the $q+1$ lines through $G$ in $\langle l_\alpha, l_\beta \rangle$. The $2q+1$ other points of $\Gamma$ lie on these lines and since every line contains $2$ or at least $q-2$ points not in $D$, we find that through $G$ there is at least one $(\leq 3)$-secant $l_1$. Consider now a point $G'\in \Gamma \setminus l_1$. Through this point there is also a $(\leq 3)$-secant $l_2$. The lines $l_1\cup l_2$ contain at most $2q+1$ points of $\Gamma$, so there is at least one $0$-point $X$ not contained in these two lines.

\end{itemize}
This shows that two $(q-1)$-secants cannot meet in a point $P'$ of $D$. Suppose now that $P'\notin D$. As above, we find for a given point $R\in D\cap l_\alpha$, at least $\frac{q}{2}-2$ $(q-1)$-secants $R R_i$, different from $l_\alpha$. But by the previous part, we know that there are no two $(q-1)$-secants through a point $R \in D$. As $\frac{q}{2}-2\geq 2$, we find a contradiction.  

\end{proof}

We now deduce a corollary that will be useful later.

\begin{gevolg}\label{q-1en3disjunct}
A $(q-1)$-secant and a $3$-secant to $D$ in $H_\infty$ cannot have a $0$-point in common.
\end{gevolg}
\begin{proof}
Let $l_\alpha$ be a $3$-secant to $D$, $l_\beta$ be a $(q-1)$-secant to $D$, and $P'=l_\alpha \cap l_\beta$ be a $0$-point. Pick $P_0 \in \Q$ and let $\alpha=\langle P_0, l_\alpha \rangle$ and $\beta=\langle P_0, l_\beta \rangle$. 
The points of $Q\cup D$ in $\alpha$ form a Fano plane: let $P'_i$, $i=1,2,3$, be the three points of $D$ on the line $l_\alpha$ and let $P_i$, $i=1,2,3$ be the corresponding affine points of $\Q$ so that $P_0P_i\cap l_\alpha =P_i'$. Since there are only three directions $P_1',P_2',P_3'$ of $D$ in $\alpha$,  we find that $\{P_1,P_3,P_2'\}$,$\{P_1,P_2,P_3'\}$ and $\{P_2,P_3,P_1'\}$ are triples of collinear points. Since also $\{P_1',P'_2,P'_3\}$ and $\{P_0,P_i,P'_i\}$, $i=1,2,3$ are triples of collinear points, we find that the points $\{P_0,P_1,P_2,P_3,P'_1,P'_2,P'_3\}$ define a Fano plane $\PG(2,2)$. 
Let $R _0$ be the point $P_1'P_2\cap P'P_0$.  Note that $R_0\notin \Q$. As the points of $\Q$ in $\beta$ form a $q$-arc, we know that there are at least two lines $R_0 R_1$ and $R_0 R_2$ in $\beta$, with $R_1,R_2 \in l_\beta \cap D$, such that both lines contain $2$ points of $\Q$. By Lemma \ref{hulplemma} we see that the lines $P_1'R_1$ and $P_1'R_2$ are both $(q-1)$-secants through $P_1'$. This gives a contradiction by Proposition \ref{q-1disjunct}.
\end{proof}

\subsection{The set $D$ of directions in $H_\infty$ is a linear set}\label{sectionlinear}

Recall that we use the notation $\tilde{P}$ for the affine point in $\Pi_2$, corresponding to the affine point $P\in \Pi_q$. 
 Let $\S'$ be the $(h-1)$-spread in the hyperplane $\tilde{H_\infty}$ of $\PG(2hk,2)$ corresponding to the points of the hyperplane $H_\infty$ of $\Pi_q$. We use the notation $\S'(P')$ for the element of $\S'$ corresponding to the point $P'\in H_\infty$. We will now show that $D$ is an $\F_2$-linear set in $H_\infty$ by showing that its points correspond to spread elements in $\tilde{H_\infty}$ intersecting some fixed $(hk-1)$-subspace of $\tilde{H_\infty}$. 

  Let $\mathscr{Q}=\Q \cup D$, $\tilde{\mathscr{Q}}=\tilde{\Q} \cup \tilde{D}$, with  $\tilde{\Q}$ the union of the points  $\tilde{P},$ with $P\in Q$, and $\tilde{D}$ the directions in $\tilde{H_\infty}$ determined by the points of $\tilde{\Q}$.

\begin{lemma}\label{lemmavlak}
Let $P_0, P_1, P_2 \in \Q$ and $P'_i=P_0P_i \cap H_\infty$, $i=1,2$.  If $P_1'P_2'$ is a $3$-secant to $D$, then the plane in $\PG(2hk,2)$ spanned by $\Tilde{P_0}$, $\Tilde{P_1}$ and $\Tilde{P_2}$ is contained in $\Tilde{\mathscr{Q}}$.
\end{lemma}
\begin{proof}
Since $P_1'P_2'$ is not a $(q-1)$-secant, we know that there is a unique point $P_3'\neq P'_1,P'_2$ in $P_1'P_2'\cap D$, and a point $P_3 \in \Q$ such that $P'_3 \in P_0P_3$. Let $\alpha$ be the plane spanned by the points $P_0, P_1$ and $P_2$. As $\alpha\cap D=\{P_1',P_2',P_3'\}$, we find that $\{P_1,P_3,P_2'\}$,$\{P_1,P_2,P_3'\}$ and $\{P_2,P_3,P_1'\}$ are triples of collinear points. As in the proof of \ref{q-1en3disjunct}, we find that these points define a Fano plane $\PG(2,2)$. 
We claim that the corresponding points $\Tilde{P_0}$, $\Tilde{P_1}$, $\Tilde{P_2}$ and $\Tilde{P_3}$ lie in a plane in $\PG(2hk,2)$. Suppose these points are not contained in a plane in $\PG(2hk,2)$, then they span a $3$-space $\beta$. 
Since $P_1'=P_0P_1\cap P_2P_3$, $\Tilde{P_0}\Tilde{P_1}$ meets $\S'(P_1')$ in a point, say $A_1$. Similarly, $\Tilde{P_2}\Tilde{P_3}$ meets $\S'(P_1')$ in a point, say $B_1$. Since $\Tilde{P_0},\Tilde{P_1},\Tilde{P_2},\Tilde{P_3}$ span a $3$-space, $A_1\neq B_1$. 
Similarly, the points $A_2=\Tilde{P_0}\Tilde{P_2}\cap \S'(P_2')$ and $B_2=\Tilde{P_1}\Tilde{P_3}\cap \S'(P_2')$ are different and span the line $A_2 B_2$. But now $A_1B_1\in \S'(\Tilde{P_1'})$ and $A_2B_2\in \S'(\Tilde{P_2'})$ are two lines in the plane $\beta\cap \Tilde{H_\infty}$, so they intersect, a contradiction since the spread elements $\S'(P_1')$ and $\S'(P_2')$ are disjoint.
\end{proof}

\begin{theorem}\label{linearset}
The set $D$ is an $\F_2$-linear set.
\end{theorem}

\begin{proof}
We will show that the set $\Tilde{\mathscr{Q}}$ of points in $\PG(2hk,2)$ forms a subspace. By Lemma \ref{lemmavlak}, we have the following property:  
if $\Tilde{P_0}$, $\Tilde{P_1}$ and $\Tilde{P_2}$ are three points in $\Tilde{\Q}$ such that the line at infinity of the plane spanned by these points corresponds to a $3$-secant in $\Pi_q$, then we know that all points of $\langle \Tilde{P_0},\Tilde{P_1},\Tilde{P_2}\rangle$ are included in $\Tilde{\mathscr{Q}}$.

Consider now a point $\tilde{P_0}\in \tilde{\Q}$ and a point $\Tilde{P_1} \in \tilde{\D}$. Let $P_1$ be the point corresponding to the spread element through $\Tilde{P_1}$ (i.e. $P_1$ is the unique point such that $\Tilde{P_1}$ is contained in $\S'(P_1)$). By Proposition \ref{q-1disjunct} we can take two $3$-secants to $D$, say $L_\alpha$ and $L_\beta$ through $P_1$ in $\Pi_q$. Let $l_\alpha$ and $l_\beta$ denote the unique line through $\Tilde{P_1}$ such that the spread elements intersecting $l_\alpha$ and $l_\beta$ correspond precisely the the points of $D$ on $L_\alpha\cup L_\beta$. Let $\alpha=\langle \tilde{P_0}, l_\alpha \rangle$ and $\beta=\langle \tilde{P_0}, l_\beta \rangle$.

Let $\{\tilde{P_1}, \tilde{P_2}, \tilde{P_3} \}= \alpha \cap \tilde{H_\infty}$ and let $\{\tilde{P_1}, \tilde{P_4}, \tilde{P_5}\} = \beta \cap \tilde{H_\infty}$. Consider an affine point $\Tilde{P}$ in $ \gamma=\langle \alpha,\beta \rangle$, $\tilde{P}\notin \alpha \cup \beta$. We want to show that $\tilde{P}$ lies in $\Tilde{\Q}$. Let $\Tilde{P'}$ be the point at infinity of the line $\Tilde{P_0} \Tilde{P}$. W.l.o.g. we suppose that $\Tilde{P'}= \Tilde{P_2}\Tilde{P_5}\cap \Tilde{P_3}\Tilde{P_4}$. Let $P_2,P_3,P_4,P_5$ be the points in $H_\infty$ corresponding to the spreadelements of $\S'$ through $\tilde{P_2}, \tilde{P_3}, \tilde{P_4}, \tilde{P_5}$. We know that $P_2P_5$ and $P_3P_4$ cannot both be $(q-1)$-secants by Proposition \ref{q-1disjunct}. So suppose that $P_2P_5$ is a $3$-secant in $\PG(2k-1,q)$. By Lemma \ref{lemmavlak}, we know that all points of the plane $\langle \Tilde{P_2}\Tilde{P_5}, \Tilde{P_0}\rangle$ lie in $\Tilde{\mathscr{Q}}$. This proves that $\Tilde{P}\in \Tilde{\Q}$, and so that $\gamma \setminus \tilde{H_\infty} \subset \tilde{\Q}$. As $\tilde{\D}$ is the set of directions determined by $\tilde{\Q}$, we also find that $\gamma \cap \tilde{H_\infty} \in \tilde{\D}$.

We conclude that all points of a $3$-space through a point $\Tilde{P_0}$ of $\Tilde{\mathcal{Q}}$, whose point set at infinity corresponds to two intersecting  $3$-secants at infinity, are contained in $\Tilde{\mathscr{Q}}$. 
 
 Now suppose that there is a $t$-space $\beta$, with $\beta \subset \tilde{\mathscr{Q}}$. By the previous part of this proof, we may assume that the points in $H_\infty$, corresponding to the spread elements intersecting $\beta \cap \Tilde{H_\infty}$, are not all contained in a single $(q-1)$-secant.
 
 If $t=hk$, then our proof is finished, so assume that $t<hk$. This implies that there exists a point $\Tilde{G}\in \Tilde{\Q}\setminus \beta$. Let $G$ be the corresponding point in $\mathcal{Q}$ in $\PG(2k,q)$, and let $\gamma=\langle \beta,\Tilde{G}\rangle$.
 We show that every point $\tilde{X}$ in $\gamma \setminus \beta$ is a point of $\tilde{\mathscr{Q}}$. 
 Suppose first that $\tilde{X}$ is a point at infinity of $\gamma \setminus \beta$, then the line $\tilde{X}\Tilde{G}$ contains an affine point $\Tilde{Y}$ of $\beta$, as $\beta$ is a hyperplane of $\gamma$. But since $\Tilde{G}$ and $\Tilde{Y}$ are points of $\Tilde{\Q}$, we find that $\tilde{X}\in \Tilde{D}\subset\tilde{\mathscr{Q}}$. 
 
 Suppose now that $\tilde{X}$ is an affine point in 
 $\gamma \setminus \beta$, and let $X$ be the corresponding point in $\PG(2k,q)$.  As the field size in $\PG(2hk,2)$ is $2$, the line $\tilde{X}\Tilde{G}$ contains $1$ extra point $\Tilde{Y}$. This point has to lie in $\beta$ and in the hyperplane at infinity, so $\Tilde{Y}\in \beta\cap \tilde{H_\infty}$. Let $l_1$ be a line through $\Tilde{Y}$  in $\beta$ corresponding to a $3$-secant, which exists since we have seen that not all points corresponding to points of $\beta \cap H_\infty$ are contained in one single $(q-1)$-secant. The plane spanned by $\Tilde{G}$ and $l_1$ is contained in $\Tilde{\mathscr{Q}}$ by Lemma \ref{lemmavlak}, and hence, since $X$ lies on the line $\Tilde{Y}\Tilde{G}$ which is contained in this plane, $X\in \Tilde{\mathscr{Q}}$. This implies that $\gamma\subseteq \mathscr{Q}$. We can repeat this argument until we find that $\tilde{\mathscr{Q}}$ is a $hk$-space in $\PG(2hk,2)$.
 \end{proof}

Note that $D$ is a scattered linear set since $|D|=q^k-1=2^{hk}-1=|\PG(hk-1,2)|$. As $D$ has rank $hk$, we find that $D$ is maximum scattered.

\begin{remark} In Lemma \ref{q-1disjunct}, we showed that the $(q-1)$-secants to $D$ were disjoint. In Theorem \ref{linearset}, we have used this to show that $D$ is a maximum scattered $\F_2$-linear set. The fact that $(q-1)$-secants to a maximum scattered $\F_2$-linear set are disjoint, is well-known (see e.g. \cite[Proposition 3.2]{italiaan}).
\end{remark}

\subsection{The set $D$ is an $\F_2$-linear set of pseudoregulus type}\label{sectionpseudo}
The proof that $D$ is of pseudoregulus type, is based on some ideas of \cite[Lemma 5 and Lemma 7]{GVML}.
\begin{lemma}\label{lemmaq+1disjointq-1k}
There are $\frac{q^k-1}{q-1}$ pairwise disjoint $(q-1)$-secants to $D$ in $\PG(2k-1,q), q>4$.
\end{lemma}
\begin{proof}
Let $K$ be the $(hk-1)$-dimensional subspace in $\PG(2hk-1,2)$ defining the $\F_2$-linear set $D$ and let $\S'$ be the $(h-1)$-spread that corresponds to the point set of $\PG(2k-1,q)$.
For every $hk$-space $Y$ through $K$ in $\PG(2hk-1,2)$, we find at least one element of $\S'$ that intersects $Y$ in a line since $D$ is maximum scattered. Every line $l$, through a point of $K$, such that $l$ lies in an element of $\S'$, defines a $hk$-space through $K$, and the number of $hk$-spaces through $K$ is $2^{hk}-1$. This implies that there are on average $2^{h-1}-1>2$ lines contained in different spread elements of $\S'$ in a $hk$-space through $K$ in $\PG(2hk-1,2)$. 

Take a $hk$-space $Y$ through $K$ with at least two lines contained in spread elements, and let $S_1$ and $S_2$ be two elements of $\S'$ that intersect $Y$ in the lines $y_1$ and $y_2$ respectively. The $(2h-1)$-space $\langle S_1,S_2\rangle$ intersects $K$ in at least a plane, as $y_1$ and $y_2$ span a $3$-space. But this implies that the line $l$ in $\PG(2k-1,q)$, corresponding with $\langle S_1,S_2\rangle$ contains at least $7$ points of $D$. This implies that $l$ is a $(q-1)$-secant of $D$, and that $\langle S_1,S_2\rangle$ intersects $K$ in a $(h-1)$-space $\alpha$ as a $(h-1)$-space contains $2^h-1=q-1$ points. 
Consider now the $h$-space $\beta=Y\cap \langle S_1,S_2 \rangle$ through $\alpha$. Since all of the $2^h+1$ $(h-1)$-spaces of $\S'$ in $\langle S_1,S_2 \rangle$ intersect $\beta$ in a point or a line, we find that there are precisely $2^{h-1}-1$ elements of $\S'$, meeting $\beta$, and so $Y$, in a line. 
Hence, this proves that a $hk$-space $Y$ through $K$, containing at least $2$ lines $y_1,y_2$ in $S_1,S_2$ respectively, contains at least $2^{h-1}-1$ lines $y_i$ in different spread elements of $\S'$. Now we prove, by contradiction, that $Y$ cannot contain more lines $y_i$ contained in a spread element. Suppose $Y$ contains another line $y_0 \subset S_0$ with $S_0 \in \S'$, then $y_0 \notin \langle S_1, S_2 \rangle$. Repeating the previous argument for $y_1$ and $y_2$ shows that there are two $(2h-1)$-spaces $\langle S_1, S_2 \rangle$ and $\langle S_0, S_1 \rangle$, both meeting $K$ in a $(h-1)$-space and so, there are two $(q-1)$-secants through $P_1\in H_\infty$, the point corresponding to the spread element $S_1$. This gives a contradiction by Proposition \ref{q-1disjunct}. 

Since the average number of lines contained in a spreadelement in a $hk$-space through $K$ is $2^{h-1}-1>2$, we find that every $hk$-space through $K$ contains exactly $2^{h-1}-1$ lines contained in a spreadelement. In particular, every line $y_i \subset S_i$, with $S_i \in \S'$ and $y_i$ through a point of $K$, defines a $hk$-space though $K$, and so a $(q-1)$-secant. So we find that every point in $D$ is contained in at least one $(q-1)$-secant. As we already proved that two $(q-1)$-secants are disjoint (see Lemma \ref{q-1disjunct}), we find $\frac{q^k-1}{q-1}$ pairwise disjoint $(q-1)$-secants in $\PG(2k-1,q)$.
\end{proof}

We will first show that the linear set is of pseudoregulus type when $k=2$. To prove this, we begin with a lemma.

\begin{lemma}\label{rechtedoor20is0}
Assume that $k=2$. Let  $l$ be a line in $H_\infty$ through two $0$-points, not on the same $(q-1)$-secant, then $l$ contains no points of $D$.
\end{lemma}
\begin{proof}
Let $l_1$ and $l_2$ be two $(q-1)$-secants in $\PG(3,q)$ and let $l$ be a line through a $0$-point of $l_1$ and through a $0$-point of $l_2$. Recall that $l_1$ and $l_2$ are disjoint by Lemma \ref{q-1disjunct}.
Every two points $(A,B)$, $A\in l_1$, $B\in l_2$, define a third point in $D$ on the line $AB$. Hence we find, since $|D|=q^2-1$, that every point $P\in D\setminus \{l_1,l_2\}$ is uniquely defined as a third point on a line, defined by two points $A$ and $B$ of $D$ in $l_1$ and $l_2$ respectively. 

Now suppose that $l$ contains a point $X\in D$, then $X$ lies on a unique line $l'$, intersecting $l_1$ and $l_2$ in precisely one point.  But then $l_1$ and $l_2$ lie in a plane spanned by $l$ and $l'$, a contradiction since $l_1$ and $l_2$ are disjoint by Lemma \ref{q-1disjunct}. 
\end{proof}

\begin{proposition}\label{thmpseudo2}
Assume that $k=2$. The $(q-1)$-secants to $D$ in $\PG(3,q)$ form a pseudoregulus. 
\end{proposition}
\begin{proof}
By Lemma \ref{lemmaq+1disjointq-1k} it is sufficient to prove that there exist $2$ lines in $\PG(3,q)$ that have a point in common with all $(q-1)$-secants to $D$.
Consider three $(q-1)$-secants $l_1$, $l_2$ and $l_3$ and let $P_i,Q_i\in l_i$, $i=1,2,3$ be the corresponding $0$-points. 
Let $l_0$ be the unique line through $P_1$ that intersects $l_2$ and $l_3$ both in a point, say $R_2=l_0\cap l_2$ and $R_3=l_0\cap l_3$ respectively. By Corollary $\ref{q-1en3disjunct}$, $R_2$ and $R_3$ cannot both belong to $\Q$, so suppose $R_2$ is a $0$-point of $l_2$ (w.l.o.g. $R_2=P_2$). We see that $l_0=P_1P_2$ is a line through two $0$-points, so $R_3$ is also a $0$-point by Corollary \ref{rechtedoor20is0}, w.l.o.g. $R_3=P_3$. By the same argument,  we see that $Q_1,Q_2$ and $Q_3$ are contained in a line, say $l_\infty$.

Now we want to show that every other $(q-1)$-secant has a $0$-point in common with both $l_0$ and $l_\infty$.
Consider an $(q-1)$-secant $l_4$, different from $l_1,l_2,l_3$, with $0$-points $P_4$ and $Q_4$. Consider now again the unique line $m$ through $P_4$ that intersects $l_1$ and $l_2$ in a point. By the previous arguments $m$ has to contain a $0$-point of $l_1$ and a $0$-point of $l_2$, so $m=l_0$,  $m=l_\infty$, $m=P_1Q_2$ or $m=Q_1P_2$. We will show that only the first two possibilities can occur, which then proves that every other $0$-point lies on $l_0$ or $l_\infty$. 
Suppose to the contrary that $m=P_1Q_2P_4$ (the case $m=Q_1P_2P_4$ is completely analogous). Then  the unique line through $Q_4$, meeting $l_1$ and $l_2$ is the line $Q_1P_2$. Consider now the unique line $m'$ through $P_4$ meeting $l_2$ and $l_3$ in a point. As we supposed that $m\neq l_0$ and $m\neq l_\infty$, we see that $P_4$ cannot lie on these lines, so $m'$ contains the points $P_4,P_2,Q_3$ or the points $P_4,Q_2,P_3$. In the former case both lines $l_0$ and $l_\infty$ are contained in the plane spanned by $m'=P_4Q_3P_2$ and $m=P_1Q_2P_4$. This implies that the disjoint lines $l_1$ and $l_2$ are contained in this plane, a contradiction. If $m'=P_4P_3Q_2$, then $m$ and $m'$ both contain $P_4$ and $Q_2$ but intersect $l_0$ in different points, a contradiction. 
We conclude that  $P_4$, and analoguously $P_4'$, is contained in the line $l_0$ or $l_\infty$. 
\end{proof}

{\color{red}}
Using the previous proposition, we will prove that for all $k$, the $\F_2$-linear set $D$ in $\PG(2k-1,q)$ is of pseudoregulus type.
\begin{theorem}\label{thmpseudo}
The $(q-1)$-secants to $D$ in $\PG(2k-1,q)$ form a pseudoregulus.
\end{theorem}
\begin{proof}
By Lemma \ref{lemmaq+1disjointq-1k} it is sufficient to prove that there exist two $(k-1)$-spaces in $\PG(2k-1,q)$ that both have a point in common with all $(q-1)$-secants to $D$.

Consider a $(q-1)$-secant $l_0$, and let $P_0$ and $P_0'$ be the 0-points on $l_0$. Let $l_i$ be a $(q-1)$-secant, different from $l_0$. The lines $l_0$ and $l_i$ span a $3$-space $\gamma$ and  since $D$ is a scattered $\F_2$-linear set, $\gamma\cap D$ is also a scattered $\F_2$-linear set. Since $\gamma$ contains $2(q-1)$ points of $D$ on the lines $l_i,l_j$ and $(q-1)^2$ points of $D$ defined in a unique way as a third point on the line $A_1 A_2$, with $A_1 \in l_0$, $A_2\in l_i$, we have that $|D\cap \gamma|=q^2-1$, and hence it is a maximum scattered linear set. By Theorem \ref{thmpseudo2}, we find that $\gamma\cap D$ is of pseudoregulus type. This means that it has transversal lines, say $m_i$ and $m_i'$, where $P_0$ lies on $m_i$ and $P_0'$ lies on $m_i'$. This holds for every $(q-1)$-secant $l_i$. Since there are exactly $\frac{q^k-1}{q-1}$ $(q-1)$-secants to $D$, which are mutually disjoint, there are exactly $2\frac{q^k-1}{q-1}$ $0$-points. We have proven that a $0$-point lies on $\frac{q^{k-1}-1}{q-1}$ lines full of $0$-points (call such lines {\em $0$-lines}) and on $\frac{q^k-1}{q-1}$ lines containing exactly $1$ other $0$-point. 

Let $A$ and $A'$ be the set of all points on the lines $m_i$ and $m_i'$ respectively. Then we will to show that $A\cup A'$ is the union of two disjoint $(k-1)$-spaces. 

Consider a line containing two $0$-points $P_1,P_2$, with $l_1$ and $l_2$ the $(q-1)$-secants through $P_1,P_2$. Then, as seen before, the intersection of the $3$-space spanned by $l_1$ and $l_2$ with $D$ is a linear set of pseudoregulus type, and hence the line $P_1P_2$ contains $2$ or $q+1$ $0$-points. This shows that every line in $\PG(2k-1,q)$ intersects $A \cup A'$ in $0,1,2$ or $q+1$ points.
This in turn implies that a plane with three $0$-lines only contains $0$-points.  
Consider now a point $P_3$ on a $0$-line through $P_0$, and consider a $0$-line $m\neq P_0P_3$ through $P_3$. If $m$ contains a point $P_4 \neq P_3$ such that $P_4P_0$ is a $0$-line through $P_0$, then we see that the plane $\langle P_0,m \rangle$ only contains $0$-points. In the other case, $M$ contains at least two $0$-points on $0$-lines through $P_0'$. In this case, all the points in the plane $\langle P_0', m \rangle$ are $0$-points, and hence the line $P_1P_0'$ is a $0$-line, a contradiction. 
So we find that every $0$-line through a $0$-point of $A$ is contained in $A$. Since every point of $A$ lies on $\frac{q^{k-1}-1}{q-1}$ $0$-lines, and $A$ contains $\frac{q^k-1}{q-1}$ $0$-points, we find that every $2$ points of $A$ are contained in a $0$-line of $A$. The same argument works for the set $A'$.
This shows that $A$ forms a subspace and likewise $A'$ forms a subspace. Since $|A|=|A'|=\frac{q^k-1}{q-1}$, these subspaces are $(k-1)$-dimensional.

\end{proof}

\vspace{0.3cm}

\subsection{There exists a suitable Desarguesian $(k-1)$-spread $\S$ in $\PG(2k-1,q)$}\label{sectiongoodspread}
Consider the scattered linear set $D\subset H_\infty$ of pseudoregulus type. Let $T_0$ and $T_\infty$ be the transversal $(k-1)$-spaces to the pseudoregulus defined by $D$ found in Theorem \ref{thmpseudo}. Now we want to show that there exists a Desarguesian $(k-1)$-spread $\S$ in $\PG(2k-1,q)$ such that $T_0,T_\infty \in \S$ and such that every other $(k-1)$-space of $\S$ has precisely one point in common with $D$. 

\begin{lemma}\label{lemmajuistespread}
There exists a Desarguesian $(k-1)$-spread $\S$ in $\PG(2k-1,q)$,  such that $T_0,T_\infty \in \S$ and such that every other element of $\S$ has precisely one point in common with $D$.
\end{lemma}
\begin{proof}
We prove this lemma using the representation of Result \ref{thmitaliaan}. 
By \cite[Theorem 3.7]{italiaan} we find that the linear sets $L_{\rho,f}$ and $L_{\rho',g}$ are equivalent if and only if $\sigma_f=\sigma_g^{\pm 1}$, where $\sigma_f$ and $\sigma_g$ are the automorphisms associated with $f$ and $g$ respectively. 
Hence, up to equivalence, we may suppose that $\rho=1$ and	 $f:\mathbb{F}_{q^k} \rightarrow \mathbb{F}_{q^k}: t\rightarrow t^{2^i}$, $\gcd(i,hk)=1$. 

Considering $U_0,U_\infty$ as $\F_{q^k}$, it follows that $D$ is equivalent to the set of points $P_u$ with  $$P_u:=(u,u^{2^i})_q,u\in\mathbb{F}_{q^k}^*.$$
The transversal spaces $T_0$ and $T_\infty$ are the point sets $T_0=\{(u,0)|u\in \mathbb{F}_{q^k}^*\}$ and $T_\infty=\{(0,u)|u\in \mathbb{F}_{q^k}^*\}$.

Consider now the set $\S_0$ of $(k-1)$-spaces $T_u$, $u\in \mathbb{F}_{q^k}^*$ with
\begin{align}
    T_u:=\{(\alpha u, \alpha u^{2^i})_{q}|\alpha\in\mathbb{F}_{q^k}^*\}.    
\end{align}

We will show that the set $\S=\S_0\cup \{T_0,T_\infty\}$ is a $(k-1)$-spread of $\PG(2k-1,q)$. 
Suppose that $P=T_{u_1}\cap T_{u_2}$, for some $u_1, u_2\notin \{0, \infty\}$, then there exists elements $\alpha_1,\alpha_2 \in \F_{q^k}^*, \mu \in \F_{q}^*$ such that

\begin{align}\label{stelselvgl}
\left \{ 
\begin{array}{ll}
\alpha_1 u_1&=\mu \alpha_2 u_2 \\
\alpha_1 u_1^{2^i}&=\mu \alpha_2 u_2^{2^i} \end{array}\right.
\end{align}
With $\mu\in \mathbb{F}_q^*$. This implies that $u_1^{2^i-1}=u_2^{2^i-1}$ or $\left(\frac{u_1}{u_2}\right)^{2^i}=\frac{u_1}{u_2}$. Hence $\frac{u_1}{u_2}\in \F_{2^i}\cap \F_{2^{hk}}$ which is $\F_2$ since $\gcd(i,hk)=1$. Since $u_1,u_2\in \F_{q^k}^*$, this implies that $u_1=u_2$, and that $T_{u_1}=T_{u_2}$. In particular, we see that $T_u \neq T_{u'}$ for $u\neq u' \in \F_{q^k}^*$. Since $T_0$ and $T_\infty$ are distinct from $T_u$ for all $u\in \F_{q^k}^*$, we obtain that $|\S|=q^k+1$.

We will now show that $T_u\cap T_0=\emptyset$ for all $u\in \F_{q^k}^*$. If $P=T_u\cap T_0$, $u\notin \{0, \infty\}$ for some $u\in \mathbb{F}_{q^k}^*$ then $P=(u',0)_q$ with $u'\in \F_{q^k}^*$ and 
\begin{align*}
\left \{ 
\begin{array}{ll}
\alpha u&=\mu  u' \\
\alpha u^{2^i}&=0 \end{array}\right.
\end{align*}
for some $\mu\in \mathbb{F}_q^*$ and $\alpha\in \mathbb{F}_{q^k}^*$. The second equality gives a contradiction since $u\neq 0 \neq \alpha$. Hence $T_u \cap T_0 = \emptyset$. It follows from a similar argument that $T_u \cap T_\infty=\emptyset$.
This shows that $\mathcal{S}$ is a spread which is Desarguesian as seen in Subsection \ref{subsectionlinear}.

\end{proof}

\begin{remark}
 In \cite[Theorem 3.11(i)]{italiaan} a geometric construction of the Desarguesian spread, found in Lemma \ref{lemmajuistespread}, using indicator sets, is given. 
\end{remark}

\subsection{The point set $\Q$ defines a translation hyperoval in the Andr\'e/Bruck-Bose plane $\mathcal{P}(S)$}\label{sectionhyperoval}
The spread $\S$ found in Lemma \ref{lemmajuistespread} defines a projective plane $\P(\S)=\Pi_{q^k}\cong \PG(2,q^k)$ by the Andr\'e/Bruck-Bose construction. The transversal $(k-1)$-spaces $T_0, T_\infty \in \S$ to the pseudoregulus associated with $D$ correspond to points $P_0, P_\infty$ contained in the line $\ell_\infty$ at infinity of $\PG(2,q^k)$.

\begin{theorem}\label{omgekeerd}
The set $\mathcal{Q}$, together with $T_0$ and $T_\infty$, defines a translation hyperoval in $\Pi_{q^k}\cong \PG(2,q^k)$.\end{theorem}
\begin{proof} 
Let $\A$ be the set of points in $\Pi_{q^k}$ corresponding to the point set $\mathcal{Q}$ of $\Pi_q$. Recall that $T_0$ corresponds to a point $P_0$ and $T_\infty$ to a point $P_\infty$, contained in the line $\ell_\infty$ of $\Pi_{q^k}$. We first show that every line in $\PG(2,q^k)$ contains at most $2$ points of the set $\mathcal{H}=\mathcal{A}\cup P_0\cup P_\infty$.

\begin{itemize}
    \item The line $\ell_\infty$ at infinity only contains the points $P_0$ and $P_\infty$.
    \item Consider a line $l\neq \ell_\infty$ through $P_0$ in $\PG(2,q^k)$. This line corresponds to a $k$-space through $T_0$ in $\PG(2k,q)$. As $P_0\in l\cap \mathcal{H}$, we have to show that this $k$-space contains at most one affine point of $\Q$. If this space would contain $2$ (or more) affine points $X_1,X_2\in \Q$, then they would define a direction of $D$ at infinity in $T_0$. But this is impossible as $T_0$ has no points of $D$, see Corollary \ref{rechtedoor20is0}. This argument also works for the lines through $P_\infty$, different from $\ell_\infty$.
    \item Consider a line $l$ through a point $P_i$, $i\notin \{0,\infty\}$ at infinity. This point $P_i$ corresponds to an element $T_i\in \S$ that intersects the pseudoregulus $D$ in a unique point $X_i$. The line $l$ corresponds to a $k$-space $\gamma$ in $\PG(2k,q)$ through $T_i$. 
    Suppose that $\gamma$ contains at least $3$ points from $\Q$, say $X,Y,Z$. By Lemma \ref{no3collinear} these points are not collinear, hence they determine at least two different points of $D$ which are contained in $T_i$, a contradiction. This proves that $\gamma$ contains at most two points of $\Q$, which implies that the line $l$ contains at most two points of $\mathcal{A}$.
\end{itemize}
Since $\mathcal{H}$ has size $q^k+2$, it follows that $\mathcal{H}$ is a hyperoval. 

Finally consider the group $G$ of elations in $\PG(2hk,2)$ with axis the hyperplane at infinity $\tilde{H}_\infty$. Since the points of $\Tilde{\mathscr{Q}}$ form a subspace, we see that $G$ acts transitively on the points of $\Tilde{\Q}$. Every element of $G$ induces an element of the group $G'$ of elations in $\PG(2,q^k)$ with axis the line $P_0P_\infty$. Hence, $G'$ acts transitively on the points of $\A$ in $\PG(2,q^k)$. This shows that $\mathcal{H}$ is a translation hyperoval.

\end{proof}

\subsection{Every translation hyperoval defines a linear set of pseudoregulus type}\label{sectionterugkeer}

 In this section, we show that the vice versa part of Theorem \ref{main} holds.

\begin{proposition} Via the Andr\'e/Bruck-Bose construction, the set of affine points of a translation hyperoval in $\PG(2,q^k)$, $q=2^h$, where $h,k\geq 2$ corresponds to a set $\Q$ of $q^k$ affine points in $\PG(2k,q)$ whose set of determined directions $D$ is an $\F_2$-linear set of pseudoregulus type.
\end{proposition}
\begin{proof} 	Consider a translation hyperoval $H\in \PG(2,q^k)$. Without loss of generality we may suppose that $H=\{(1,t,t^{2^i})_{q^k} | t\in \mathbb{F}_{q^k}\} \cup \{(0,1,0)_{q^k},(0,0,1)_{q^k}\}$ with $\gcd(i,hk)=1$. The set of affine points of $H$ corresponds to the set of points $H'=\{(1,t,t^{2^i})_{q} \in \mathbb{F}_q \oplus \mathbb{F}_{q^k} \oplus \mathbb{F}_{q^k} | t\in \mathbb{F}_{q^k}\}$ in $\PG(2k,q)$ (for more information about the use of these coordinates for $H$ and $H'$, see \cite{johnsarageertrui}).  The determined directions in the hyperplane at infinity $H_\infty: X_0=0$ have coordinates $(0,t_1-t_2,t_1^{2^i}-t_2^{2^i})_q$ where $t_1,t_2 \in \mathbb{F}_{q^k}$. So the set $D=\{(0,u,u^{2^i} ) _q| u\in \mathbb{F}_{q^k}   \}$ is precisely the set of directions determined by the points of $H$. By Result \ref{thmitaliaan} we find that this set of directions $D$ is an $\F_2$-linear set of pseudoregulus type in the hyperplane $H_\infty$.
	
\end{proof}

We will now show that every line in $\PG(2k-1,q)$ intersects the points of the linear set $D$ in $0,1,3$ or $q-1$ points. 

\begin{proposition}
	Let $D$ be the set of points of an $\mathbb{F}_2$-linear set of pseudoregulus type in $\PG(2k-1,q)$, $q=2^h$, $h>2$, $k\geq 2$. Then every line of $\PG(2k-1,q)$ meets $D$ in $0,1,3$ or $q-1$ points.
\end{proposition}
\begin{proof} 
We use the representation of Result \ref{thmitaliaan} for the points of $D$. Let $R_1=(u_1,f(u_1))_q$ and $R_2=(u_2,f(u_2))_q$, $u_1,u_2 \in U_0$, be two points of $D$ not on the same line of the pseudoregulus, so the vectors $\langle u_1 \rangle$ and $\langle u_2 \rangle$ in $V(k,q)$ are not an $\mathbb{F}_q$-multiple (in short $\langle u_1 \rangle_q \neq \langle u_2 \rangle_q$). Recall that $f$ is a invertible semilinear map with automorphism $\sigma \in Aut(\mathbb{F}_q)$, $Fix(\sigma)=\{0,1\}$. A third point $R_3=(u_3,f(u_3))_q\in D$ is contained in $R_1R_2$ if and only if there are $\mu,\lambda\in \F_q$ such that
 \begin{align*}
     \left \{ 
     \begin{array}{ll}
        u_1+\lambda u_2&=\mu u_3 \\
     f(u_1)+\lambda f(u_2) &= \mu f(u_3) 
     \end{array}
       \right.
 \Leftrightarrow
    \left \{ \begin{array}{ll}
     f(u_1)+\lambda^\sigma f(u_2)&=\mu^\sigma f(u_3) \\
    f(u_1)+\lambda f(u_2) &= \mu f(u_3)
 \end{array} \right.\\
\Leftrightarrow
 \left \{ \begin{array}{ll}
     u_1+\lambda u_2&=\mu u_3 \\
    (\lambda^\sigma - \lambda) f(u_2) &= f((\mu-\mu^{\sigma^{-1}})u_3)
\end{array} \right.
\Leftrightarrow
    \left \{ \begin{array}{ll}
     u_1+\lambda u_2&=\mu u_3 \\
    (\lambda^\sigma - \lambda)^{\sigma^{-1}} u_2 &= (\mu-\mu^{\sigma^{-1}})u_3
\end{array} \right.
 \end{align*}
As $R_2$ and $R_3$ lie on different $(q-1)$-secants to $D$, we have that $\langle u_2 \rangle_q \neq \langle u_3 \rangle_q$. It follows that $\lambda^\sigma - \lambda=\mu-\mu^{\sigma^{-1}}=0$, so $\lambda,\mu\in Fix(\sigma)=\{0,1\}$. We find that there is only one solution of this system, such that $R_1\neq R_3$ (i.e. $\langle u_1 \rangle_q \neq \langle u_3 \rangle_q$), namely when $\lambda=\mu=1$. Hence, given two points $R_1,R_2$ in $D$, there is an unique point $R_3\in D\cap R_1R_2$, different from $R_1$ and $R_2$.

\end{proof}

\section{The generalisation of a characterisation of Barwick and Jackson}\label{veralgemening}

Using Theorem \ref{main}, we are now able to generalise the following result of Barwick-Jackson which concerns translation hyperovals in $\PG(2,q^2)$ (\cite{BJeven}). 

 \begin{result}\cite[Theorem 1.2]{BJeven} \label{BJ}Consider $\PG(4,q)$, $q$ even, $q>2$, with the hyperplane at infinity denoted by $\Sigma_\infty$. Let $\C$ be a set of $q^2$ affine points, called $\mathcal{C}$-points and consider a set of planes called $\mathcal{C}$-planes which satisfies the following:
\begin{itemize}
\item[(A1)] Each $\C$-plane meets $\C$ in a $q$-arc.
\item[(A2)] Any two distinct $\C$-points lie in a unique $\C$-plane.
\item[(A3)] The affine points that are not in $\C$ lie on exactly one $\C$-plane.
\item[(A4)] Every plane which meets $\C$ in at least $3$ points either meets $\C$ in $4$ points or is a $\C$-plane.
\end{itemize}
Then there exists a Desarguesian spread $\mathcal{S}$ in $\Sigma_\infty$ such that in the Bruck-Bose plane $\mathcal{P}(\mathcal{S})\cong \PG(2,q^2)$, the $\C$-points, together with $2$ extra points on $\ell_\infty$ form a translation hyperoval in $\PG(2,q^2)$.

\end{result}

\begin{remark} At two different points, the proofs of \cite{BJeven} are inherently linked to the fact that they are dealing with hyperovals in $\PG(2,q^2)$. In \cite[Lemma 4.1]{BJeven} the authors show the existence of a design which is isomorphic to an affine plane, of which they later need to use the parallel classes. In \cite[Theorem 4.11]{BJeven}, they use the Klein correspondence to represent lines in $\PG(3,q)$ in $\PG(5,q)$. Both techniques cannot be extended in a straightforward way to $q^k$, $k>2$.
\end{remark}

The following Proposition shows that a set of $\C$-planes as defined by Barwick and Jackson in \cite{BJeven} (using $\PG(2k,q)$ instead of $\PG(4,q)$) satisfies the conditions of Theorem \ref{main}.
\begin{proposition} \label{prop}Consider $\PG(2k,q)$, $q$ even, $q>2$, with the hyperplane at infinity denoted by $\Sigma_\infty$. Let $\C$ be a set of $q^k$ affine points, called $\mathcal{C}$-points and consider a set of planes called $\mathcal{C}$-planes which satisfies the following:
\begin{itemize}
\item[(A1)] Each $\C$-plane meets $\C$ in a $q$-arc.
\item[(A2)] Any two distinct $\C$-points lie in a unique $\C$-plane.
\item[(A3)] The affine points that are not in $\C$ lie on exactly one $\C$-plane.
\item[(A4)] Every plane which meets $\C$ in at least $3$ points either meets $\C$ in $4$ points or is a $\C$-plane.
\end{itemize} Then $\C$ determines a set of $q^k-1$ directions $D$ in $\Sigma_\infty$ such that every line of $\Sigma_\infty$ meets $D$ in $0,1,3$ or $q-1$ points.
\end{proposition}
\begin{proof} As before, we call the points that are not contained in $\Sigma_\infty$ {\em affine} points. Note that all $\C$-points are affine. Since every two $\C$-points lie on a $\C$-plane which meets $\C$ in a $q$-arc, we have that no three $\C$-points are collinear.

Let $P_0$ be a $\C$-point and let $D_0$ be the set of points of the form $P_0P_i\cap \Sigma_\infty$, where $P_i\neq P_0$ is a point of $\C$. We first show that every line meets $D_0$ in $0,1,3$ or $q-1$ points. Let $M$ be a line of $\Sigma_\infty$ containing $2$ points of $D_0$, say $R_1'=P_0R_1\cap \Sigma_\infty$, $R_2'=P_0R_2\cap \Sigma_\infty$, where $R_1,R_2\in \C$. Then $\langle M,P_0\rangle$ contains at least $3$ points of $\C$, and hence, by (A4), either it is a $\C$-plane or it contains exactly $4$ points of $\C$. If $\langle M,P_0\rangle$ is a $\C$-plane, it contains $q$ points of $\C$ forming a $q$-arc, and hence, $M$ contains $q-1$ points of $D_0$. Now suppose that $\langle M,P_0\rangle$ contains exactly $4$ $\C$-points, then $M$ contains $3$ points of $D_0$. 

Now let $P_1\neq P_0$ be a point of $\C$ and let $D_1$ be the set of points of the form $P_1P_i\cap \Sigma_\infty$, where $P_i\neq P_1$ is a point of $\C$. We claim that $D_0=D_1$. Let $P_1'=P_0P_1\cap \Sigma_\infty$. We see that $P_1'\in D_0\cap D_1$. Consider a point $P_2'\neq P_1'$ in $D_0$, then $P_0P_2\cap \Sigma_\infty=P_2'$ for some $P_2\in \C$. Consider the plane $\pi=\langle P_0,P_1,P_2\rangle$. 

Suppose first that $\pi$ is not a $\C$-plane, then, by (A4), $\pi$ contains exactly one extra point, say $P_3$ of $\C$. The lines $P_0P_1$ and $P_2P_3$ lie in $\pi$ and hence, meet in a point $Q$. By $(A2)$, there is a $\C$-plane $\mu$ through $P_0P_1$, and likewise, there is a $\C$-plane $\mu'$ through $P_2P_3$. Since $\pi$ is not a $\C$-plane, $\mu$ and $\mu'$ are two distinct  $\C$-planes through $Q$. By (A3) his implies that $Q$ is a point of $\Sigma_\infty$. Likewise, $P_0P_2\cap P_1P_3$ and $P_0P_3\cap P_1P_2$ are points of $\Sigma_\infty$. It follows that $D_0\cap \pi=D_1\cap \pi$. This argument shows that for all points $R\neq P_1'\in D_0$ such that $\langle P_0,P_1,R\rangle$ is not a $\C$-plane, we have that $R\in D_1$. Now $P_0P_1$ lies on a unique $\C$-plane, say $\nu$. Let $\nu \cap \Sigma_\infty=L$, then we have shown that $\langle P_0,P_1,R\rangle$ is not a $\C$-plane as long as $R\in \Sigma_\infty$ is not on $L$. We conclude that $D_0\setminus L=D_1\setminus L$.

Now assume that $D_0\neq D_1$ and let $X$ be a point in $D_1$ which is not contained in $D_0$. Then $X\in L$ and $P_1X$ contains a point $Y\neq P_1\in \C$. Consider a point $P_4'\in D_1$, not on $L$, then $P_1P_4'$ contains a point $P_4\neq P_1$ of $\C$. Since $P_4'\in D_1\setminus L$, $P_4'\in D_0$ so the line $P_4'P_0$ contains a point $P_5\neq P_1$ of $\C$.

The plane $\langle P_1,P_4',X\rangle$ is not a $\C$-plane since otherwise, the points $P_1$ and $Y$ of $\C$ would lie in two different $\C$-planes. This implies that $\langle P_1,P_4,X\rangle$ which contains the $\C$-points $P_1,P_4,Y$ contains exactly one extra point of $\C$, say $P_6$. Denote $P_1P_6\cap \Sigma_\infty$ by $P'_6$. We see that there are exactly 3 points of $D_1$ on the line $P_4'X$, namely $P_4',X$ and $P_6'$.

Now $P'_6$ is a point of $D_1$, not on $L$, so $P'_6\in D_0$. Hence, there is a point $S\neq P_0\in \C$ on the line $P_0P'_6$.

If $\langle P_4',P'_6,P_0\rangle$ is not a $\C$-plane, then, since it contains $P_0,P_5,S$ of $\C$ it contains precisely $3$ points  of $D_0$ at infinity. These are the points $P_4',P_6'$ and one other point, say $T$, which needs to be different from $X$ by our assumption that $X\notin D_0$. That implies that $T$ is not on $L$, and hence, $T\in D_1$. This is a contradiction since we have seen that the only points of $D_1$ on $P_4'X$ are $P_4',X$ and $P_6'$. Now if $\langle P_4',P_6,P_0\rangle$ is a $\C$-plane, we find $q-1$ points of $D_0$ on $P_4'X$, all of them are not on $L$. Hence, we find $q-1$ points of $D_1$ on $P_4'X$, not on $L$. This is again a contradiction since $P_4'X$ has only the points $P_4'$ and $P_6'$ of $D_1$ not on $L$.

This proves our claim that $D_0=D_1$.
Since $P_1$ was chosen arbitrarily, different from $P_0$, and $D_0=D_1$, we find that the set $D$ of directions determined by $\C$ is precisely the set $D_0$. The statement now follows from the fact that a line meets $D_0$ in $0,1,3$ or $q-1$ points.

\end{proof}

Proposition \ref{prop} shows that the set $\C$ satisfies the criteria of Theorem \ref{main}. Hence, we find the following generalisation of Result \ref{BJ}.
 \begin{theorem}\label{main2} Consider $\PG(2k,q)$, $q$ even, $q>2$, with the hyperplane at infinity denoted by $\Sigma_\infty$. Let $\C$ be a set of $q^k$ affine points, called $\mathcal{C}$-points and consider a set of planes called $\mathcal{C}$-planes which satisfies the following:
\begin{itemize}
\item[(A1)] Each $\C$-plane meets $\C$ in a $q$-arc.
\item[(A2)] Any two distinct $\C$-points lie in a unique $\C$-plane.
\item[(A3)] The affine points that are not in $\C$ lie on exactly one $\C$-plane.
\item[(A4)] Every plane which meets $\C$ in at least $3$ points either meets $\C$ in $4$ points or is a $\C$-plane.
\end{itemize}
Then there exists a Desarguesian spread $\mathcal{S}$ in $\Sigma_\infty$ such that in the Bruck-Bose plane $\mathcal{P}(\mathcal{S})\cong \PG(2,q^k)$, the $\C$-points, together with $2$ extra points on $\ell_\infty$ form a translation hyperoval in $\PG(2,q^k)$.
\end{theorem}

\end{document}